\documentclass[12pt, a4paper, reqno, oneside]{amsart}
\usepackage{amsmath,amsthm,amssymb,mathrsfs}
\usepackage{color,hyperref,graphicx,comment}
\usepackage[utf8]{inputenc}
\usepackage[all]{xy}
\usepackage{mathtools}
\usepackage{enumitem}
\usepackage{faktor}
\usepackage{tikz-cd}
\usepackage[all]{xy} 
\usepackage{upgreek}
\usepackage{bbold} 
\usepackage{stmaryrd} 
\usepackage{cleveref} 

\allowdisplaybreaks

\usepackage{geometry}
\numberwithin{equation}{section}

\theoremstyle{plain}
  \newtheorem{thm}{Theorem}[section]
  \newtheorem{lem}[thm]{Lemma}
  \newtheorem{prop}[thm]{Proposition}
  \newtheorem{cor}[thm]{Corollary}
  \newtheorem*{thm*}{Theorem}  
  \newtheorem*{thmm*}{Main Theorem} 
\theoremstyle{definition}
  \newtheorem{defn}[thm]{Definition}
    \newtheorem{assump}[thm]{Assumption}
  \newtheorem{rmk}[thm]{Remark}
  \newtheorem{obs}[thm]{Observation}
  \newtheorem*{ack*}{Acknowledgements}
  \newtheorem*{ref*}{Reference}
  \newtheorem{ex}[thm]{Example}
  \newtheorem*{ft*}{Fact}
\theoremstyle{plain}

\newcommand\pl{\partial}
\newcommand\iip[2]{\left\langle{#1},{#2}\right\rangle}

\newcommand\af{\alpha}

\newcommand\om{\omega}
\newcommand\gm{\gamma}
\newcommand\Om{\Omega}

\newcommand\ta{\theta}
\newcommand\ld{\lambda}
\newcommand\sm{\sigma}
\newcommand\Gm{\Gamma}
\newcommand\Sm{\Sigma}

\newcommand\Dt{\Delta}

\newcommand\vph{\varphi}

\newcommand\w{\wedge}

\newcommand\BR{\mathbb{R}}
\newcommand\BC{\mathbb{C}}

\newcommand\BZ{\mathbb{Z}}
\newcommand\BN{\mathbb{N}}

\newcommand\BS{\mathbb{S}}
\newcommand\BM{\mathbb{M}}

\newcommand\bx{\mathbf{x}}

\newcommand\bz{\mathbf{z}}

\newcommand\bp{\mathbf{p}}

\newcommand\CA{\mathcal{A}}
\newcommand\CC{\mathcal{C}}
\newcommand\CI{\mathcal{I}}

\newcommand\CD{\mathcal{D}}
\newcommand\CH{\mathcal{H}}
\newcommand\CL{\mathcal{L}}

\newcommand\CE{\mathcal{E}}
\newcommand\CG{\mathcal{G}}

\newcommand\SL{\mathscr{L}}

\newcommand\SH{\mathscr{H}}

\DeclareMathOperator{\re}{Re}
\DeclareMathOperator{\im}{Im}
\DeclareMathOperator{\Int}{int}

\DeclareMathOperator{\spt}{spt}
\DeclareMathOperator{\reg}{Reg}
\DeclareMathOperator{\sing}{Sing}

\title{Uniqueness in the Plateau problem for \\ calibrated currents}
\author{Bryan Dimler and Chen-Kuan Lee}
\address{Department of Mathematics \\ University of California, Irvine \\ Irvine, CA 92697 \\ USA}
\email{bdimler@uci.edu}
\address{Department of Mathematics\\University of Notre Dame\\Notre Dame, IN 46556\\USA}
\email{clee36@nd.edu}

\begin{document}
\begin{sloppypar}

\begin{abstract}
We show that every compactly supported smoothly calibrated integral current with connected $C^{3,\alpha}$ boundary is the unique solution to the oriented Plateau problem for its boundary data. The same holds true for compactly supported ``continuously calibrated" integral flat chains. This is proved as a consequence of the boundary regularity theory for area-minimizing currents and a unique continuation argument in the spirit of Frank Morgan. In codimension one, the argument yields a sufficient condition for uniqueness in the oriented Plateau problem expressed in terms of the regularity of the calibrating form.
\end{abstract}

\maketitle
\section{\textbf{Introduction}}

The Plateau problem asks whether a given boundary bounds a minimal surface with least area. In the early 1930's, Douglas \cite{Do31} and Rad\'o \cite{Rad30} gave the positive answer for Jordan curves in $\BR^3$, and their work was generalized to Riemannian manifolds in 1948 by Morrey \cite{Mor48}. Later, Federer and Fleming \cite{FF60} introduced the notion of \emph{area-minimizing integral currents}\footnote{See Section 2 for the formal definitions.}, a well-known generalization of minimal surfaces, to formulate and address the \emph{oriented} Plateau problem. Due to its broad applicability across all dimensions and codimensions, Federer and Fleming's theory marks a major success in establishing existence results.

A natural question is whether solutions to the Plateau problem are unique. In general, this is hard since there are counterexamples, even in simple cases. For example, if one considers the Plateau problem in $\BR^2$ with boundary given by the four corners of a square, two distinct solutions are given by pairs of parallel line segments that coincide with the square's sides. One can even produce smooth connected boundaries admitting at least three solutions---or even a continuum of solutions (see e.g. \cite{Nit68, Mor76}). Fortunately, uniqueness is known to be a generic property for both the oriented and unoriented Plateau problems, where the latter case is interpreted in the sense of flat chains modulo 2 (see \cite{Mor78, Mor82, CMMS}). 

The first generic uniqueness theorems were due to Morgan in \cite{Mor78}, and were proved for surfaces in $\BR^3$. The proof relies fundamentally on a unique continuation argument at the boundary (see \Cref{obs1}). Shortly after, he extended his results to elliptic integrands in arbitrary dimensions, as well as to all codimensions in the special case of area-minimizing flat chains modulo 2 (see \cite{Mor81, Mor82}). Due to their reliance on Allard's boundary regularity \cite{All75}, Morgan's original results were confined to the Euclidean setting, required uniform convexity assumptions on the boundary, and, for area-minimizing integral $k$-currents, were restricted to codimension one. Despite this, he knew that his results could be extended in both cases provided there was a suitable boundary regularity theory available (see e.g. \cite[Remark after Theorem 7.1]{Mor82}). Recently, Caldini, Marchese, Merlo, and Steinbr\"uchel proved Morgan's observation \cite[Theorem 1.3]{CMMS} as a consequence of the boundary regularity theory of De Lellis, De Philippis, Hirsch and Massaccesi in \cite{CDHM}. Since their proof relies on a more general boundary regularity framework, it does not require uniform convexity assumptions on the boundary like Morgan's theorems do.

Given a specific boundary, it remains unclear whether it uniquely bounds an area-minimizing integral current or a flat chain modulo 2. Uniqueness in the Plateau problem can be proved only in restricted cases, such as for surfaces in $\BR^3$ bounded by special Jordan curves \cite{Rad30, Nit73} and for hypersurfaces in $\BR^n$ satisfying suitable conditions \cite{Lin87, NS25, Ha77}. In this paper, we prove uniqueness in the oriented Plateau problem under a natural geometric condition---without imposing any restrictions on dimension or codimension---using Morgan's unique continuation arguments and the boundary regularity theory in \cite{CDHM}. Specifically, we show that \emph{every} compactly supported, \emph{smoothly calibrated} integral current with connected $C^{3,\alpha}$ boundary uniquely solves the oriented Plateau problem for its boundary data. Our result also recovers several known cases. The following assumptions are crucial.
\begin{assump}\label{A1}
    Fix $n \geq 3$ and $1 \leq k \leq n-1$. Let $W \subset \subset U \subset \BR^n$ be open connected sets. We will always assume that $\Gm \subset W$ is a $(k-1)$-dimensional closed oriented connected embedded $C^{3,\alpha}$ submanifold of $\BR^n$, that $T \in \CI_{k, c}(\BR^n)$ (i.e. integral $k$-current in $\BR^n$ with compact support) is area-minimizing in $U$ with $\spt T \subset W$, and that $\partial T = [[\Gamma]]$.
\end{assump}

We can now state our main theorem, which is a special case of \Cref{main}.

\begin{thm}[Uniqueness in the Oriented Plateau Problem] \label{UOPP}
    Let $U = \BR^n$ and suppose that $T \in \CI_{k,c}(\BR^n)$ and $\Gm$ are as in \Cref{A1}. Let $\vph \in \CE^k(\BR^n)$ (i.e. smooth $k$-form in $\BR^n$) be a calibration. Assume that $T$ is calibrated by $\vph$ in $\BR^n$. If $T^\prime \in \CI_{k, c}(\BR^n)$ is \emph{globally} area-minimizing with $\pl T^\prime = [[\Gm]]$, then $T^\prime = T$ in $\BR^n$.
\end{thm}
\begin{rmk}
    \Cref{UOPP} can be extended to integral flat $k$-chains assuming the calibrating flat $k$-cochain $\af$, identified with a $k$-form $\vph_\af$, is merely continuous (\Cref{flatchain}). Doing so, we obtain a sufficient condition for uniqueness in the oriented Plateau problem in codimension one expressed in terms of the regularity of $\vph_\alpha$ (\Cref{plateau}).
\end{rmk}

For general ambient manifolds, \Cref{UOPP} and \Cref{main} do not necessarily hold. To illustrate the issue, let $\mathbb{T}^2 \coloneqq \BR^2 / \BZ^2$ be the flat torus. Consider the points $p = [(0, 0)] = [(1, 0)]$ and $q = [(\frac{1}{2}, 0)]$ on $\mathbb{T}^2$. Then the line segment $\ell_1$ from $(0, 0)$ to $(\frac{1}{2}, 0)$ and the line segment $\ell_2$ from $(1, 0)$ to $(\frac{1}{2}, 0)$ are calibrated by $dx$ and $-dx$, respectively. Both $[[\ell_1]]$ and $[[\ell_2]]$ define area-minimizing $1$-currents on $\mathbb{T}^2$ with boundary $-[[p]] + [[q]]$, yet they are clearly distinct. This non-uniqueness arises from the nontrivial topology of $\mathbb{T}^2$. Specifically, we have $H_1(\mathbb{T}^2; \BR) = \BR^2$. However, with a mild topological assumption we can recover our theorem.

\begin{rmk} \label{ambient}
    With only cosmetic changes, \Cref{UOPP} and all the other results in this paper concerning area-minimizing currents and integral flat chains continue to hold when the ambient manifold $M$ is complete, without boundary, analytic, and satisfies $H_k(M; \BR) = 0$.
\end{rmk}

Before we continue, let us introduce the Cauchy problem since it is the foundation of our proof. Let $\Om \subset \BR^k$ be a bounded $C^{1,1}$ domain and let $L : H^1(\Om; \BR^{n-k}) \rightarrow H^{-1}(\Om; \mathbb{R}^{n-k})$ be the second-order \emph{principally diagonal}\footnote{That is, a linear system having no leading-order coupling.} elliptic partial differential operator of divergence form:
\begin{equation}
    (Lu(\bx))^\sm = \pl_i(a^{ij}(\bx)u_{x^j}^\sm(\bx)) + b_\gm^{\sm s}(\bx) u_{x^s}^\gm(\bx) + c_\gm^\sm(\bx) u^\gm(\bx) \text{ for } \sm = 1,\ldots, n-k, \label{L1}
\end{equation}
where $A := (a^{ij})_{1 \leq i,j\leq k} \in C^{0,1}(\overline{\Om}; \BR^{k\times k})$ is a positive definite symmetric matrix with eigenvalues in $[\lambda, \Lambda]$ for $0 < \lambda < \Lambda < \infty$ and $b_\gm^{\sm s}, c_\gm^\sm \in \CL^\infty(\Om)$. The \emph{Cauchy problem} asks whether there exists a unique $u \in H^1(\Om; \BR^{n-k})$ solving 
\begin{equation}\label{cauchy1}
    \begin{cases}
       Lu(\bx) = 0 \text{ in } \Om \\
        u(\bx) = 0 \text{ and } 
        \pl_\nu u(\bx) = 0 \text{ on } \Gm,
    \end{cases}
\end{equation}
where $\Gm$ is a relatively open portion of $\pl \Om$, $\nu$ is the unit outer normal to $\Gm$, and $L$ is given by \eqref{L1}. Since any solution $u \in H^1(\Om; \BR^{n-k})$ to \eqref{cauchy1} has zero trace on $\Gm$, it is $C^{1,\alpha}$ up to $\Gm$ by elliptic regularity so that $u$ and $\pl_\nu u$ continuously vanish along $\Gm$. By boundary Carleman estimates, if $u$ solves \eqref{cauchy1} then $u \equiv 0$ in $\Om$ (see \cite[Remark 2]{A} and \cite{AKS, AE97}).

Since uniqueness in the Cauchy problem follows from Carleman estimates, it is intimately tied to the unique continuation properties of elliptic operators. For this reason, it is often called \emph{unique continuation from Cauchy data}. We say that a function $u \in H^1(\Om; \BR^{n-k})$ has the \emph{strong unique continuation property} (SUCP) in $\Om$ if $u \equiv 0$ whenever there is a point $\bx_0 \in \Om$ at which $u$ vanishes to infinite order \cite[see (1.7) p. 246]{GL86}. The SUCP for solutions $u$ to $Lw = 0$ in $\Om$ with $L$ in \eqref{L1} was first proved by Aronszajn, Krzywicki, and Szarski in \cite{AKS} using Carleman estimates (see \cite[Remark 3]{A} and \cite{GL86, GL87, KT01}). The condition that $L$ is principally diagonal is necessary, since there are counterexamples to SUCP and uniqueness in \eqref{cauchy1} when we allow for leading-order coupling (see e.g. \cite{Pl63}). Moreover, the results in \cite{A, AKS, GL86, KT01} are essentially sharp, because for each $\alpha \in (0,1)$ there are elliptic scalar operators of divergence and non-divergence forms with leading coefficient functions $C^{0,\alpha}$ for which SUCP fails (see \cite{Pl63, Mi74}).

The observation by Morgan we will need is that, by the fundamental theorem of calculus (FTC), the difference $w := u-v$ of two solutions $u,v \in C^2(\Om; \BR^{n-k})$ to the \emph{minimal surface system} (MSS) on a domain $\Om \subset \BR^k$ (see Section 3.1) satisfies $Lw = 0$ in $\Om$ for a principally diagonal operator $L$ of the form \eqref{L1}. Applying uniqueness in the Cauchy problem gives:
\begin{obs}[{cf. \cite[Lemma 7.2]{Mor82} }]\label{obs1} 
    Any two $C^{2}$ minimal submanifolds that are tangent along a relatively open portion of their boundaries are equal as submanifolds.
\end{obs} 
In the setting of area-minimizing integral $k$-currents, Morgan points out that unique continuation from Cauchy data is inherited from the MSS whenever there is sufficient regularity (see \cite[Remark 5.4]{Mor81} and \cite[Lemma 7.2]{Mor82}). However, in general area-minimizing currents tend to exhibit singularities. When the codimension is one, it is well known that, away from its boundary, the support of an area-minimizing integral $k$-current in Euclidean space is an analytic hypersurface outside of a relatively closed subset of Hausdorff dimension at most $k-7$. We refer to this set as the \emph{interior singular set}. For higher codimension area-minimizing integral $k$-currents, Almgren \cite{Alm1} proved that the interior singular set has Hausdorff dimension at most $k-2$ (see also \cite{DS14, DS16a, DS16b}).

Boundary regularity is more subtle. In codimension one, several results are known (see e.g. \cite{All75, HS1}). On the other hand, in higher codimension our understanding remained limited for a long time. Only recently did De Lellis, De Philippis, Hirsch, and Massaccesi \cite{CDHM} establish the first general boundary regularity theorem without restrictions on the codimension or the ambient manifold. In particular, they showed that if $\Gm \subset \BR^n$ and $T \in \CI_{k,c}(\BR^n)$ are as in \Cref{A1}, then the set of \emph{boundary regular points} is open and dense in $\Gm = \spt \pl T$ so that $\spt T$ can be locally represented as an embedded $C^{3,\alpha}$ minimal submanifold (see \Cref{thm1}). 

Let $\Gm \subset \BR^n$, let $T,T^\prime \in \CI_{k,c}(\BR^n)$ be as in \Cref{A1}, and suppose $T$ is smoothly calibrated by $\vph \in \CE^k(\BR^n)$. \Cref{UOPP} now follows from a couple of simple observations. The first is that $T^\prime$ is also calibrated by $\vph$ (\Cref{areaminimpliescal}). Then, the first cousin principle (\Cref{FirstCousin}), together with a rigidity theorem for calibrated $k$-planes (\Cref{IPL}) and the boundary regularity theory, implies that the \emph{true} tangent spaces for $\spt T$ and $\spt T^\prime$ coincide on a relatively open portion of $\Gm$. This allows us to apply \Cref{obs1}, from which uniqueness follows.

The paper is organized as follows. Section 2 introduces the basic notation and definitions concerning the MSS, as well as area-minimizing and calibrated integral currents. In Section 3, we briefly recall the interior regularity theory for Lipschitz stationary solutions to the MSS \eqref{MSS}, and then outline Morgan's unique continuation arguments in \cite[Theorem 7.1]{Mor82} adapted to our setting. We have stated a (slightly) more general version of uniqueness in the Cauchy problem than that found in \cite[Lemma 7.2]{Mor82} using the partial regularity theorem in \cite[Theorem 3.7]{BD2}. In addition, we have included a survey of the regularity theory for area-minimizing integral currents since it is fundamental to our proof. Section 4 contains a general version of \Cref{UOPP} (i.e. \Cref{main}), along with a short discussion on the sharpness of the hypotheses in \Cref{UOPP} and the extension to integral flat chains. In Section 5, we cover \Cref{ambient} in depth using a certain special Lagrangian sphere to demonstrate the importance of the topological constraint on the ambient space. 

Due to its relevance to the present paper and the lack of a precise reference, we have included a proof of the following folklore result in Appendix A: \emph{If a compactly supported area-minimizing integral current is ``extendable", then it is unique}. We compare this with \Cref{UOPP}, and recover two well-known cases: uniqueness in the oriented Plateau problem for regular area-minimizing cones bounded by their link (\Cref{cone}), and uniqueness of restrictions of $C^{3,\alpha}$ graphical hypersurfaces over convex domains \cite[Theorem 2]{Ha77}.

\begin{ack*}
   The authors thank Rick Schoen for familiarizing them with the problem, Jacek Rzemieniecki for pointing out the first cousin principle (i.e. \Cref{FirstCousin}), and Zhenhua Liu for bringing Federer's work \cite{Fed74} to their attention, thus motivating \Cref{flatchain}. The authors extend a special thanks to Frank Morgan for suggesting the sharp hypotheses needed on the calibrating form, and for providing additional examples of non-uniqueness. 
   
   The first author thanks Connor Mooney for his support and guidance, as well as Zihui Zhao for an informative conversation on the Cauchy problem and the boundary regularity theory for area-minimizing currents at the RMMC Summer School, 2025. The second author thanks Nick Edelen for his many insightful comments, and Chung-Jun Tsai for pointing out the obstruction for general ambient spaces. He also thanks National Taiwan University, where part of this work was carried out, for their hospitality. The first author was supported by C. Mooney's NSF CAREER Grant DMS-2143668, as well as NSF RTG Grant DMS-2342135. 
\end{ack*}
\section{\textbf{Preliminaries}}
This section serves as a short primer on the MSS and area-minimizing integral currents. For a thorough account, see \cite{GM12, Sim18, HL82}.

\vspace{-9pt}
\subsection{Minimal surface system}
Let $\Om \subset \BR^k$ be a domain (not necessarily bounded) with Lipschitz boundary and suppose $u \in C_{\text{loc}}^{0}(\overline{\Om}; \BR^{n-k}) \cap C_{\text{loc}}^{0,1}(\Om; \BR^{n-k})$. Then its graph $\Sm_u$ is a $k$-dimensional Lipschitz submanifold of $\BR^n$ with boundary 
\[\pl \Sm_u := \{(\bx,u(\bx)) : \bx \in \pl \Om \subset \BR^k\} \subset \BR^n.\] 
Define 
\begin{equation*}
    F: \BR^{(n-k)\times k} \rightarrow \BR \text{ by } F(\bp) := \sqrt{\det( I + \bp^T\bp)}.
\end{equation*}
Then the $k$-dimensional surface area of $\Sm_u$ is given by the formula
\begin{equation*} \label{areafunctional}
    \CA(Du) := \int_{\Sm_u} \, d\CH^k = \int_\Om \sqrt{\det g(Du(\bx))}\, d\bx = \int_\Om F(Du(\bx)) \, d\bx,
\end{equation*}
where $\CH^k$ is the $k$-dimensional Hausdorff measure and the induced metric $g$ is $\CH^k$-a.e. defined by
\begin{equation*}
    g(Du(\bx)) := I + Du(\bx)^TDu(\bx) = (\delta_{ij} + u_{x^i}(\bx) \cdot u_{x^j}(\bx))_{1 \leq i,\,j \leq k}. 
\end{equation*}
The submanifold $\Sm_u$ can be identified with a stationary varifold $V_u:= \underline{v}(\Sm_u, 1)$ (see Section 2.3.1) if and only if $u$ is a \emph{stationary solution} to the MSS in $\Om$:
\begin{equation}
    \begin{cases}
        \partial_i(F_{p_i^\gm}(Du(\bx))u_{x^j}^\gm(\bx) - F(Du(\bx))\delta_{ij}) = 0 \text{ for } j = 1,\ldots, k \\
         \pl_i(F_{p_i^\sm}(Du(\bx))) = 0 \text{ for } \sm = 1,\ldots, n-k,
    \end{cases} \label{MSS}
\end{equation}
where $DF(\bp) := (F_{p_i^\sigma}(\bp))_{1 \leq i \leq k}^{1 \leq \sigma \leq n-k}$ is viewed as a map $\BR^{(n-k) \times k} \rightarrow \BR^{(n-k) \times k}$. The first equation in \eqref{MSS} is called the \emph{inner variation system} for $\CA$, while the second is called the \emph{outer variation system}, or \emph{Euler--Lagrange system}, for $\CA$. When $k = n-1$ (i.e. $\Sm_u$ has codimension one), the outer variation equation reduces to the \emph{minimal surface equation} (MSE), which is the quasi-linear divergence form scalar equation with coefficients
\begin{equation}
   a^{ij}(Du(\bx)) := \frac{1}{\sqrt{1 + |Du(\bx)|^2}}\Bigg(\delta_{ij} - \frac{u_{x^i}(\bx) u_{x^j}(\bx)}{\sqrt{1 + |Du(\bx)|^2}}\Bigg). \label{MSE}
\end{equation}

When $u \in C_{\text{loc}}^2(\Om; \BR^{n-k}) \cap C_{\text{loc}}^{0}(\overline{\Om}; \BR^{n-k})$, the MSS is equivalent to the following quasi-linear elliptic system in non-divergence form:
\begin{equation}
    g^{ij}(Du(\bx))u_{x^i x^j}^\sigma(\bx) = 0 \text{ for each } \sigma = 1,\ldots, n-k \text{ in } \Om, \label{MSS2}
\end{equation}
where we have set $g^{-1} := (g^{ij})_{1 \leq i,j \leq k}$. Note that $g^{-1}$ is positive definite with eigenvalues in $[\lambda, \Lambda] \subset (0,\infty)$ for some numbers $\lambda, \Lambda$ depending on the Lipschitz constant for $u$. Any solution $u$ solving \eqref{MSS2} is called a \emph{classical solution} to the MSS. For a given Lipschitz stationary solution $u$ to \eqref{MSS}, we will write $\sing(u) \subset \Om$ for the set of points at which $u$ fails to be $C^1$ (i.e. the \emph{singular set} for $u$). We set $\reg(u) \coloneqq \Om \setminus \sing(u)$.

\subsection{Integral currents}
For fixed $n \in \BN$, we define $\Lambda^1(\BR^n) := (\BR^n)^*$. That is, $\Lambda^1(\BR^n)$ is the space of linear functionals $\ell: \BR^n \rightarrow \BR$. For each $k \geq 2$, we will write $\Lambda^k(\BR^n)$ for the space of alternating $k$-linear functions (i.e. $k$-covectors). The space of smooth $k$-forms on $\BR^n$ is denoted by $\CE^k(\BR^n) \coloneqq C^\infty(\BR^n; \Lambda^k(\BR^n))$. We will write $\CD^k(\BR^n)$ for the space of smooth $k$-forms with compact support, which is a subspace of $\CE^k(\BR^n)$ but equipped with the usual locally convex topology. The space of $k$-vectors is denoted by $\Lambda_k(\BR^n)$. A $k$-vector is called \emph{simple} if $\xi = v_1 \w v_2 \w \cdots \w v_k$, where $v_i \in \Lambda_1(\BR^n)$ for all $i = 1, \ldots, k$ and $\{v_1,\ldots, v_k\}$ forms a linearly independent set. 

Note that we can identify the oriented linearly independent set $\langle v_1, \ldots, v_k \rangle$ with an oriented $k$-plane in $\BR^n$ passing through the origin. In the same vein, such a $k$-plane in $\BR^n$ can be represented by the simple $k$-vector formed as a wedge product of all elements in its oriented basis. That is, a simple $k$-form is equivalent to an oriented $k$-plane passing through the origin, up to a scalar factor. We will therefore write 
\[G(k, \BR^n) := \{\xi \in \Lambda_k(\BR^n):  |\xi| = 1 \text{ and } \xi \text{ is simple}\}\]
for the Grassmanian of oriented $k$-planes in $\BR^n$ through the origin. The \emph{comass norm of $\vph \in \mathcal{\CE}^k(\BR^n)$ at $p \in \BR^n$} is then given by 
\begin{equation}
    ||\vph||_p := \sup_{\xi \in G(k,\BR^n)} \langle \vph_p, \xi \rangle, \label{cmassp}
\end{equation}
where $\langle \cdot, \cdot \rangle$ is the dual pairing\footnote{Formally, $\langle \vph_p,\xi \rangle := \vph_p(\xi)$ for all $\xi \in \Lambda_k(\BR^n)$.} for $\Lambda^k(\BR^n)$ and $\Lambda_k(\BR^n)$. 
The \emph{comass norm of $\vph \in \CE^k(\BR^n)$}, denoted by $||\vph||$, is the supremum in \eqref{cmassp} over $p \in \BR^n$. We say $\vph$ has \emph{comass one} if $||\vph||=1$.

The space $\mathcal{D}_k(\BR^n)$ will represent the continuous dual space of $\CD^k(\BR^n)$ (i.e. the \emph{$k$-currents}). For any $T \in \mathcal{D}_k(\BR^n)$, we define $\pl T \in \mathcal{D}_{k-1}(\BR^n)$ (i.e. the \emph{boundary of $T$}) by \[\partial T(\vph) := T(d\vph) \text{ for all } \vph \in \CD^{k-1}(\BR^n).\] For $T \in \CD_k(\BR^n)$, the \emph{mass of $T$} is given by
\begin{equation*}
    \mathbb{M}(T) := \sup_{\vph \in \CD^k(\BR^n)}\{T(\vph): ||\vph|| \leq 1 \}.
\end{equation*}
If $U \subset \BR^n$ is an open set, we can define the \emph{mass of $T$ in $U$}, denoted by $\BM_U(T)$, by instead taking the supremum over those $\vph$ with $\spt \vph \subset U$. We say $T$ has \emph{finite mass} if $\BM(T) < \infty$, and $T$ has \emph{locally finite mass} if $\BM_W(T) < \infty$ for every open $W \subset\subset \BR^n$. If both $T$ and $\pl T$ have (locally) finite mass, then $T$ is called (\emph{locally}) \emph{normal}.
\begin{defn}[Integral Current] \label{integralcurrent}
    If $T \in \mathcal{D}_k(\BR^n)$, we say that $T$ is a \emph{locally rectifiable integral $k$-current} (abbreviated as \emph{integral $k$-current}) if $T$ is locally normal and \begin{equation*}
        T(\vph) = \int_{M_T} \langle \vph_p ,\xi_T(p) \rangle \theta_T(p) \, d\mathcal{H}^k(p) \text{ for all } \vph \in \CE^k(\BR^n), 
    \end{equation*} where $M_T \subset \BR^n$ is a $k$-rectifiable Borel set, $\theta_T \in \CL_{\text{loc}}^1(M_T; d\mathcal{H}^k)$ is a positive integer valued function on $\BR^n$, and $\xi_T: M_T \rightarrow \Lambda_k(\BR^n)$ is a $\CH^k$-measurable function such that for $\mathcal{H}^k$-a.e. $p \in M_T$ we have $\xi_T(p) = e_1 \w \cdots \w e_k$, where $\{e_1,\ldots,e_k\}$ is an orthonormal basis for the approximate tangent space $T_p M_T$. It is common to write $T = \underline{\tau}(M_T, \ta_T, \xi_T)$.
\end{defn}

\begin{rmk}
    Federer and Fleming's boundary rectifiability theorem (cf. \cite[Chapter 6, Theorem 6.3]{Sim18}) shows that if $T$ is an integral $k$-current, then $\pl T$ is an integral $(k-1)$-current.
\end{rmk}

The function $\theta_T$ is the \emph{multiplicity function} for $T$ and $\xi_T$ is called the \emph{orientation} for $T$. We will write $\|T\| \coloneqq \CH^k \llcorner \ta_T \llcorner M_T$ for the \emph{mass measure} of $T$. Note that $\BM_U(T) = \|T\|(U)$ for every Borel set $U \subset \BR^n$. Moreover, $\spt T = \spt \|T\|$, where $\spt \|T\|$ is the support of $\|T\|$ in the usual sense of Radon measures. We shall write $\CI_k(\BR^n)$ for the set of all integral $k$-currents. The subset of currents in $\CI_k(\BR^n)$ with compact support is denoted by $\CI_{k,c}(\BR^n)$. 

\begin{ex}[Lipschitz Submanifolds]
    Suppose $\Sm \subset \BR^n$ is a $k$-dimensional oriented embedded Lipschitz submanifold with $\CH^k(\Sigma) + \CH^{k-1}(\partial \Sigma) < \infty$ and $\CH^k$-a.e. defined orientation $\xi_\Sm : \Sm \rightarrow \Lambda_k(
\BR^n)$. Then there is a unique integral $k$-current $[[\Sm]] \in \CI_k(\BR^n)$ associated to $\Sm$ defined by taking $M_T = \Sm$, $\theta_T = \theta_\Sm \equiv 1$, and $\xi_T = \xi_\Sm$ in \Cref{integralcurrent}.
\end{ex}

Federer and Fleming show that the homology of the chain complex of compactly supported currents is isomorphic to the singular homology \cite[Theorem 5.11]{FF60}. As a consequence, we have:
\begin{lem} \label{closed=exact}
    If $T \in \CI_{k, c}(\BR^n)$ satisfies $\pl T = 0$, then $T = \pl S$ for some $S \in \CI_{k+1, c}(\BR^n)$. If $T$ is a normal $k$-current with compact support and satisfies $\pl T = 0$, then $T = \pl S$ for some normal $(k+1)$-current $S$ with compact support.  
\end{lem}

\subsection{Area-minimizing integral currents}
Of particular interest in the present paper are the area-minimizing integral $k$-currents.
\begin{defn} \label{areamin}
    Let $U \subset \BR^n$ be an open set. An integral $k$-current $T \in \CI_k(\BR^n)$ is \emph{area-minimizing in $U$} if for every open $W \subset\subset U$ we have \begin{equation*}
        \BM_W(T) \leq \BM_W(T + S) \text{ for all } S \in \CI_{k, c}(\BR^n) \text{ with } \pl S = 0 \text{ and } \spt S \subset W.
    \end{equation*} We say $T$ is \emph{globally area-minimizing} if $U = \BR^n$.
\end{defn}

Throughout the paper, we will often consider the case when $T$ is a globally area-minimizing integral current and has non-zero boundary $\pl T := [[\Gm]]$, where $\Gm \subset \BR^n$ satisfies \Cref{A1}. If $T$ has compact support, then $T$ has the least mass among all compactly supported integral currents with boundary $[[\Gm]]$. If $T$ is not compactly supported, then $T$ has the least mass among all ``compact perturbations'' of its support.

\subsubsection{Rectifiable varifolds}
Let $M \subset \BR^n$ be a $k$-rectifiable Borel set and $\ta \in \CL^1_{loc}(M; d\CH^k)$. A \emph{rectifiable $k$-varifold} $\underline{v}(M, \ta)$ is the equivalence class of all pairs $(M^\prime, \ta^\prime)$, where $M^\prime$ is $k$-rectifiable with $\CH^k(M \Dt M^\prime) = 0$ and $\ta = \ta^\prime$ $\CH^k$-a.e. on $M \cap M^\prime$. In particular, given an integral current $T$, the varifold associated with $T$ is just the equivalence class $\underline{v}(M_T, \ta_T)$, i.e. forgetting the orientation $\xi_T$. A rectifiable varifold is called \emph{stationary} if the first variation vanishes. Loosely speaking, stationary rectifiable varifolds can be thought of as Lipschitz minimal submanifolds (not necessarily oriented). The varifold associated with an area-minimizing current is stationary (see \cite[Chapter 7, Lemma 1.2]{Sim18}). Thus, area-minimizing integral currents can be heuristically thought of as oriented Lipschitz minimal submanifolds having the least surface area (i.e. mass) with respect to their boundaries. We refer the interested reader to \cite[Chapter 4]{Sim18} for more on rectifiable varifolds.

\subsubsection{Densities}
Let $U \subset \BR^n$ be an open set, and $p \in U$. For $T \in \CI_{k}(\BR^n)$ that is area-minimizing in $U$ and satisfies $\spt T \subset U$, the \emph{density of $T$ at $p$} is defined by \begin{equation}
    \Theta_T(p) = \lim_{r \rightarrow 0} \frac{\|T\|(B^n_r(p))}{\om_k r^k} \label{dense}
\end{equation} whenever the latter limit exists. Here, $\om_k$ denotes the volume of the unit $k$-ball. Whenever $p \notin \spt(\pl T)$, the existence of this limit is always guaranteed by the monotonicity formula for stationary varifolds \cite[Chapter 4, Equation 3.8]{Sim18}. The monotonicity formula also shows that $\Theta_T(p)$ is an upper semi-continuous function of $p$. We finally note that $\Theta_T(p) = \ta_T(p)$ for $\CH^k$-a.e. $p \in \spt T$ (see \cite[Chapter 3, Remark 1.8]{Sim18}). Therefore, $T$ has the \emph{canonical representation} \begin{equation} \label{cano}
    T = \underline{\tau}(\spt T, \Theta_T, \xi_T).
\end{equation}

\subsection{Smoothly calibrated integral currents}

Let $\vph \in \CE^k(\BR^n)$ have comass one. For $p \in \BR^n$, we write \begin{equation*}
    \CG_p(\vph) \coloneqq \{\xi \in G(k, \BR^n): \langle \vph_p, \xi \rangle = 1\}
\end{equation*} for the collection of planes where $\vph_p$ achieves its maximum. Such $\CG_p(\vph)$ is called the \emph{contact set of $\vph$ at $p$}. We also define \begin{equation*}
    \CG(\vph) \coloneqq \bigcup_{p \in \BR^n} \mathcal{G}_p(\vph).
\end{equation*}

\begin{defn}
    A smooth $k$-form $\vph \in \CE^k(\BR^n)$ is called a \emph{calibration} if $\vph$ has comass one and $d \vph = 0$. Let $U \subset \BR^n$ be an open set. An integral $k$-current $T \in \CI_k(\BR^n)$ is said to be \emph{calibrated by $\vph$ in $U$} if $\xi_T(p) \in \CG_p(\vph)$ for $\|T\|$-a.e. $p \in U$.
\end{defn}

Calibrated currents achieve their mass when acting on the calibration, and are area-minimizing.

\begin{lem}[{cf. \cite[Chapter II, Lemma 3.5]{HL82}}] \label{calibrated}
    Let $\vph \in \CE^k(\BR^n)$ be a calibration and $T \in \CI_k(\BR^n)$. Then \[(T \llcorner W)(\vph) \leq \BM_W(T)\] for every open $W \subset\subset \BR^n$, and the equality holds if and only if $T$ is calibrated by $\vph$ in $W$.
\end{lem}

\begin{lem}[{cf. \cite[Chapter II, Corollary 4.5]{HL82}}] \label{calibratedimpliesmin}
    Fix $1 \leq k \leq n-1$. Let $U \subset \BR^n$ be an open set, and let $\vph \in \CE^k(\BR^n)$ be a calibration. Suppose that $T \in \CI_k(\BR^n)$ is calibrated by $\vph$ in $U$. Then $T$ is area-minimizing in $U$.
\end{lem}

There are many examples of calibrations and calibrated currents, the easiest example being the current associated to the coordinate plane $\BR^k \times \{0\} \subset \BR^{k} \times \BR^{n-k}$ calibrated by $\vph(\bx) := dx^1 \wedge \cdots \wedge dx^k$, $\bx := (x^1, \ldots, x^n) \in \BR^n$. We provide a few historically important examples, but refer the reader to \cite{HL82} for a full treatment.

\begin{ex}[Lipschitz Stationary Graphs] \label{graph_calibrated}
   Let $\Om \subset \BR^{n}$ be a bounded $C^{2}$ mean convex domain. Suppose $u \in C^\infty(\Om) \cap C^{0}(\overline{\Om})$ is a Lipschitz stationary solution to the MSE \eqref{MSE}, whose existence for fixed continuous boundary data is guaranteed by \cite[Theorem 13.6]{Gi84}. Then the volume form on its graph $\Sm_u$ is given by 
   \begin{equation*}
       dvol_{\Sm_u} =  \frac{(-1)^{n}}{\sqrt{1 + |Du|^2} } dx^1 \wedge \cdots \wedge dx^n + \sum_{i = 1}^{n} \frac{(-1)^i u_{x^i}}{\sqrt{1 + |Du|^2} } \, dx^1 \wedge \cdots \wedge \widehat{dx^i} \wedge \cdots \wedge dx^n \wedge dy.
   \end{equation*} 
   The $n$-form $\vph(\bx,y)$ defined by extending the right-hand side of the equality above to $U := \Om \times \BR$ is a calibration and $[[\Sm_u]]$ is calibrated by $\vph$ in $U$. That $\vph$ has comass one can be checked readily, and $d\vph \equiv 0$ in $U$ since $u$ solves \eqref{MSE} in $\Om$. It follows that if $u$ solves \eqref{MSE}, then $[[\Sm_u]]$ is area-minimizing in $U$. If $\Om$ is convex, then by the \emph{convex hull property} \cite[Corollary 2]{Ha77}, $[[\Sm_u]]$ is globally area-minimizing (see \cite[Theorem 2]{Ha77}). A topological hypothesis on the boundary is necessary to ensure global uniqueness (cf. \cite{HLL87}). Since there are higher codimension stationary solutions to \eqref{MSS} over the unit ball which are not stable critical points of $\CA$ (see e.g. \cite[Theorem 5.1]{LO1}), this example does not extend to \eqref{MSS}.
\end{ex}

\begin{ex}[Complex Submanifolds in $\BC^n$] \label{complexsubmfd}
    We say that $\Sm \subset \mathbb{C}^n$ is a $k$-dimensional \emph{complex submanifold} (real dimension $2k$), with or without boundary, if $T_p \Sm$ is a $k$-dimensional complex vector subspace of $\mathbb{C}^n$ for each $p \in \Sm$. Using the identification $\mathbb{C}^n \cong \mathbb{R}^n + \sqrt{-1} \, \mathbb{R}^n$, it is common to identify $\mathbb{C}^n$ with $\mathbb{R}^{2n}$ in the coordinates \[\bz := (z^1,\ldots, z^n) = (x^1, \ldots, x^n, y^1, \ldots, y^n), \text{ } z^k := x^k + \sqrt{-1} \, y^k.\] If we define the symplectic form \begin{equation}
    \omega(\bz) := \frac{1}{2 \sqrt{-1}} \sum_{j = 1}^n dz^j \wedge d \overline{z}^j =\sum_{j = 1}^n dx^j \wedge dy^j, \label{symplectic}
\end{equation}
then 
\[\vph(\bz) := \frac{\om^k(\bz)}{k!} := \frac{1}{k!}\underbrace{\om(\bz) \wedge \cdots \wedge \om(\bz)}_{k \text{ times}}\]
is a closed complex $k$-form. By Wirtinger's inequality \cite[Section 1.8.2]{Fed69}, $\vph$ is a calibration, and $[[\Sm]]$ is calibrated by $\vph$ whenever $\Sm$ is a $k$-dimensional complex submanifold of $\BC^n$. Common examples include the graph current $[[\Sm_f]]$ in $\BC^2$ where $f: \Om \subset \BC \rightarrow \BC$ is holomorphic and $\Om$ is a domain or, more generally, $[[Z(f)]]$ where $Z(f)$ is the zero set of a holomorphic function $f: \Om \times \BC \rightarrow \BC$ having no critical points in $Z(f)$.
\end{ex}

\begin{ex}[Special Lagrangian Submanifolds in $\BC^n$]
    As in \Cref{complexsubmfd}, we identify $\BC^n \cong \BR^{2n}$ and let $\om$ be the symplectic form \eqref{symplectic}. A real $n$-dimensional submanifold $\Sm \subset \BC^n$ is called a \emph{Lagrangian submanifold} if $\om \big|_\Sm = 0$. Let $d\bz \coloneqq dz^1 \w \cdots \w dz^n = \re(d\bz)+ \im(d\bz)$ be a holomorphic $n$-form on $\BC^n$. A Lagrangian submanifold $\Sm \subset \BC^n$ is called \emph{special} if, in addition, $\im(d \bz)\big|_\Sm = 0$. In this case, $\vph(\bz) = \re(d \bz)$ (as a real $n$-form) is a calibration, and $[[\Sm]]$ is calibrated by $\vph$ whenever $\Sm$ is a special Lagrangian submanifold in $\BC^n$.

    For any $C^2$ function $u: \Om \subset \BR^n \rightarrow \BR$ defined on a domain $\Om$, the graph $\Sm_{Du} \subset \BR^{2n} \cong \BC^n$ of $D u$ is a Lagrangian submanifold in $\BC^n$ (see \cite[Chapter III, Lemma 2.2]{HL82}). Any such $u$ is called the \emph{potential function} of $\Sm_{Du}$. If, in addition, $u$ satisfies the \emph{special Lagrangian equation} \begin{equation*}
        \sum_{i = 1}^n \arctan(\ld_i(D^2u)) = \vartheta,
    \end{equation*} where $\vartheta \in (-\frac{n \pi}{2}, \frac{n \pi}{2})$ is a constant and $\ld_i(D^2 u)$ denote the eigenvalues of the Hessian matrix $D^2 u$, then $\Sm_{Du}$ is a special Lagrangian submanifold in $\BC^n$, and $\vartheta$ is called the \emph{phase} of $\Sm_{Du}$.
\end{ex}

\begin{ex}[Lawson--Osserman Cone]\label{LOC}
    Take $\SH: \mathbb{S}^3 \rightarrow \mathbb{S}^2$ to be the Hopf map:
\begin{equation*}
    \SH(z^1,z^2) = (2\overline{z}^1 z^2, |z^1|^2 - |z^2|^2) 
\end{equation*}
for $(z^1,z^2) \in \mathbb{C}^2$ with $|z^1|^2 + |z^2|^2 = 1$. After making the identifications $\mathbb{C}^2 \cong \mathbb{R}^4$ and $\mathbb{C} \times \mathbb{R}\cong \mathbb{R}^3$, we define the \emph{Lawson--Osserman cone} (LOC), denoted by $\CC$, to be the graph of $\SL: \BR^4 \rightarrow \BR^3$ given by
\begin{equation}
    \SL(\bx) := \frac{\sqrt{5}}{2}|\bx|\SH\left(\frac{\bx}{|\bx|}\right).
\end{equation} 
It is a 4-dimensional stationary cone in $\mathbb{R}^7$ with an isolated singularity at the origin, with associated integral $4$-current $[[\CC]]$. The LOC was the first example of a singular Lipschitz stationary graph \cite[Theorem 7.1]{LO1}. In \cite[Section IV.3]{HL82}, it was shown that the LOC is calibrated by the \emph{coassociative $4$-form} on $\BR^7$:
\[\vph(\bx) := dx^{1234} - dx^{67} \w (dx^{12}- dx^{34})+ dx^{57} \w (dx^{13} + dx^{24}) - dx^{56} \w (dx^{14} - dx^{23}),\]
where $dx^{i_1 i_2 \ldots i_k}:= dx^{i_1} \wedge dx^{i_2} \wedge \cdots \wedge dx^{i_k}$ for $1 \leq i_1 < i_2 < \cdots < i_k \leq 7$.
\end{ex}

\section{\textbf{Unique continuation from Cauchy data}}

For completeness and ease of reference, we provide an outline of Morgan's unique continuation arguments in \cite[Theorem 7.1]{Mor82} adapted to our setting.

\subsection{Lipschitz stationary solutions}
First, we consider uniqueness in the Cauchy problem for Lipschitz stationary solutions to \eqref{MSS}. The only difference is that we do not assume any interior regularity. Nevertheless, this is no problem due to the interior regularity theory. 
\begin{prop}[Uniqueness in the Cauchy Problem]\label{cauchy2}
   Suppose $u \in C_{\text{loc}}^{0}(\overline{\Om}; \mathbb{R}^{n-k}) \cap C_{\text{loc}}^{0,1}(\Om; \BR^{n-k})$ solves \eqref{MSS} in a Lipschitz domain $\Omega$ (not necessarily bounded) and that
    \begin{equation}
        u = \phi \text{ and } \pl_\nu u = \Phi \text{ on } \Gm \label{cauchy3}
    \end{equation}
    on a relatively open strictly convex $C^{2,\alpha}$ patch $\Gm \subset \pl \Om$, where $\nu$ is the unit outer normal to $\Gm$, $\phi \in C_{\text{loc}}^{2,\alpha}(\Gm; \BR^{n-k}$), and $\Phi \in C_{\text{loc}}^{1,\alpha}(\Gm;\BR^{(n-k)\times k})$. Let $v \in C_{\text{loc}}^{0}(\overline{\Om}; \BR^{n-k}) \cap C_{\text{loc}}^{0,1}(\Om; \BR^{n-k})$ be another solution to \eqref{MSS} in $\Om$ with Cauchy data \eqref{cauchy3}. Then $u \equiv v$ in $\Om$. When $k = n-1$, we can drop the strict convexity of $\Gm$.
\end{prop}
\begin{rmk}\label{mark}
    Due to \cite{A,AKS}, for stationary solutions it is natural to loosen \Cref{obs1} to require that $\Gm$ is $C^{1,1}$ and that the solutions are locally $C^{1,1}$ up to $\Gm$.
\end{rmk}

Interior regularity for the MSS follows from the theory for \emph{outer stationary maps} (i.e. solutions to the outer variation system). Assume only that $u: \Om \subset \BR^k \rightarrow \BR^{n-k}$ is outer stationary and $u \in C_{\text{loc}}^1(\Om; \BR^{n-k})$. Fix $\bx_0 \in \Om$ and choose $r > 0$ so that $B_r^k(\bx_0) \subset \subset \Om$. An important fact is that $\CA$ is \emph{strongly rank-one convex} when restricted to Lipschitz functions \cite[Lemma 6.7: the lemma, though stated for $k = 2$, holds for arbitrary
$k$]{Ti21}. Therefore, we can differentiate the equation in $B_r^k(\bx_0)$ to obtain a quasi-linear system with $C^{0}$ coefficients satisfying the \emph{Legendre--Hadamard ellipticity condition} \cite[Definition 3.36]{GM12}. The classical $W^{2, p}$ theory for linear elliptic systems (see e.g. \cite{Mor1, GM12}) shows that $u \in C^{1, \af}(B_r^k(\bx_0); \BR^{n-k})$.\footnote{A simple illustration of the argument for uniformly convex functionals is outlined at the beginning of \cite[Chapter 8]{GM12}, though a similar argument works for strongly rank-one convex functionals.} Morrey's Theorem \cite[Theorem 6.8.1]{Mor1} then gives $u \in C^\om(B_r^k(\bx_0); \BR^{n-k})$ (i.e. the analytic functions). As a result, $C^1$ stationary solutions to \eqref{MSS} are analytic in the interior. In general, this cannot be improved due to the LOC. The outer variation system, \eqref{MSS}, and \eqref{MSS2} are equivalent at regular points. In particular, when $k = n - 1$, the MSS reduces to the MSE, since Lipschitz solutions are always $C^1$, and hence analytic in the interior by the De Giorgi--Nash--Moser Theorem \cite[see Chapter 8]{GT01}.

\begin{proof}[Proof of \Cref{cauchy2}]
When $k = n - 1$, a standard computation using the FTC (see e.g. \cite[Lemma 1.26, p. 37]{CM1}) shows that if $u, v \in C^\om(\Om) \cap C^{0}(\overline{\Om})$ are Lipschitz stationary solutions to \eqref{MSE} in $\Om$, then their difference $w \coloneqq u-v$ solves a uniformly elliptic divergence form scalar equation with analytic coefficients in $\Om$. By boundary Schauder estimates, if $u$ and $v$ have Cauchy data \eqref{cauchy3} then $u,v \in C_{\text{loc}}^{2,\alpha}(U \cap \overline{\Om})$ for some neighborhood $U$ of $\Gm$ in $\BR^n$. Uniqueness in the Cauchy problem for the MSE then follows from the classical uniqueness theory applied to $w$ on $U \cap \Om$ and the SUCP. 

Suppose now that $1 \leq k \leq n - 2$ and that $\Gm$ is strictly convex. By Allard's boundary regularity \cite{All75}, we again have $u, v \in C_{\text{loc}}^{2,\alpha}(U \cap \overline{\Om}; \BR^{n-k})$ for some neighborhood $U$ of $\Gm$ in $\BR^n$ (see \cite[note below Theorem 2.3]{LO1}). Let $w$ be the difference function defined above. While the MSS \eqref{MSS} is a diagonal divergence form system, it is not clear from the divergence form structure that $w$ solves $Lw = 0$ in $\Om \cap U$ with $L$ in \eqref{L1}. Since $u,v \in C_{\text{loc}}^{2,\alpha}(U \cap \overline{\Om}; \BR^{n-k})$, we can pass to the non-divergence form system \eqref{MSS2} near $\Gm$. The FTC, as applied in \cite[p. 1082]{Mo15}, shows that $w$ satisfies $Lw = 0$ in $U \cap \Om$ for $L$ in \eqref{L1} with $c_\gm^\sm \equiv 0$. Therefore, $w \equiv 0$ in $\Om \cap U$ by uniqueness in the Cauchy problem.

Since Lipschitz stationary solutions can exhibit interior singularities when $k \geq 4$ (see e.g. \cite{LO1, B79, FC1, HNY1, XYZ, BD3}), the concern is that $\sing(w)$ may disconnect $\Om$. However, this does not occur. By \cite[Theorem 3.7]{BD2}, $\sing(u)$ and $\sing(v)$ are relatively closed in $\Om$ and have Hausdorff dimension at most $k-4$. Applying this result along with \cite[Lemma 6.3]{Mo15}, we conclude that $\reg(u) \cap \reg(v) \subset \reg(w)$ is open, dense, and connected in $\Om$. In particular, $w \equiv 0$ in $\Om$ by the SUCP and continuity of $w$, completing the proof of \Cref{cauchy2}. 
\end{proof}

\begin{ex}\label{hopfdata}
    The function $\SL$ in \Cref{LOC} is the unique solution to \eqref{cauchy3} in $B_1^4(0) \subset \BR^4$ with $u = \frac{\sqrt{5}}{2} \SH$ and $\pl_\nu u = \frac{\sqrt{5}}{2}\pl_\nu \SH$ on $\pl B_1^4(0)$. It is unknown whether $\SL$ is the unique solution to the Dirichlet problem for \eqref{MSS} on $B_1^4(0)$ with boundary data $\SH$ on $\pl B_1^4(0)$.
\end{ex}

An interesting phenomenon is the discrepancy between the parametric and non-parametric area-minimization problems in high codimension. For example, Mooney and Savin \cite[Theorem 1.2]{MS24} showed that large interior singular sets can form away from the boundary when minimizing area over high codimension Lipschitz graphs. This implies that, in general, the solution to the oriented Plateau problem \emph{does not} seek to be graphical when we impose graphical boundary data (cf. \Cref{graph_calibrated}). We show that, despite this peculiarity, uniqueness in the oriented Plateau problem for calibrated currents can be proved using \Cref{cauchy2}.

\subsection{Area-minimizing currents}
We prove unique continuation from Cauchy data for area-minimizing currents with $C^{3,\alpha}$ boundary after surveying the relevant regularity theory. The proofs of \Cref{UCP1} and \Cref{main} are then a simple application of the regularity theory and \Cref{obs1}. We will assume that $T \in \CI_k(\BR^n)$ is area-minimizing in $U \subset \BR^n$ and $\Gm = \spt \pl T$ unless otherwise stated.

\begin{defn}[Interior Regular Point]\label{regpti}
    A point $p \in \spt T \setminus \Gm$ is called an \emph{interior regular point} for $T$ if there is an $r > 0$ such that $B_r^n(p) \cap \Gm = \emptyset$ and a $k$-dimensional connected oriented embedded $C^1$ submanifold $\Sm \subset B_r^n(p)$ without boundary in $B_r^n(p)$ such that $T \llcorner B_r^n(p) = Q[[\Sm]]$, where $Q \in \BN$ (so $\spt (T \llcorner B_r^n(p)) = \Sm$). The set of interior regular points is denoted by $\reg_i(T)$, and the set of \emph{interior singular points} is defined by $\sing_i(T) := \spt T \setminus (\Gm \cup \reg_i(T))$. 
\end{defn}

\begin{rmk} \label{integerdensity}
    Due to \eqref{dense} and \eqref{cano}, we note that $\theta_T(p) = \Theta_T(p) = Q \in \BN$ when $p \in \reg_i(T)$. Such $Q$ may vary depending on $p$; however, the Constancy theorem (cf. \cite[Chapter 6, Theorem 2.41]{Sim18}) shows that $\Theta_T$ is locally constant on $\reg_i(T)$. Therefore, $\Theta_T$ is constant on each connected component of $\reg_i(T)$. 
\end{rmk}

\begin{rmk} \label{analytic}
    Let $p \in \reg_i(T \llcorner U)$. Then there is an $r > 0$ such that $B_r^n(p) \subset \subset U$, $B_r^n(p) \cap \Gm = \emptyset$, and $T \llcorner B_r^n(p) = Q[[\Sm]]$. In addition, $Q[[\Sm]]$ is stationary when identified with its corresponding varifold $V_\Sm := \underline{v}(\Sm, Q)$. Choosing $r$ smaller if necessary, $\Sm$ can be represented as the graph $\Sm = \Sm_u$ of some analytic stationary solution $u :\Omega \subset \BR^k \rightarrow \BR^{n-k}$ to \eqref{MSS}, where $\Om$ is a bounded, simply connected domain. 
\end{rmk}

When combined with a covering argument, analyticity yields an identity lemma for the set of interior regular points. The requirement of a shared boundary is essential. We leave the details of the proof to the reader.

\newpage
\begin{lem} \label{analyticconti}
    Let $S, S^\prime \in \CI_k(\BR^n)$ with \begin{enumerate}
        \item $\pl S = \pl S^\prime$.
        \item $\reg_i(S)$ and $\reg_i(S^\prime)$ are analytic.
    \end{enumerate} Suppose that $M$ is a connected component of $\reg_i(S)$ such that $M$ and $\reg_i(S^\prime)$ make contact at a point $x_0 \in M \cap \reg_i(S^\prime)$ of infinite order (see \cite[(1.7) in p. 246]{GL86} for a precise definition). Then $M = \reg_i(S^\prime \llcorner M)$ as oriented submanifolds of $\BR^n$.
\end{lem}

Defining boundary regular points is more delicate. In \Cref{regpt} below, we assume that $\Gm \subset \BR^n$ satisfies the hypotheses in \Cref{A1}.
\begin{defn}[Boundary Regular Point]\label{regpt}
    A point $p \in \Gm$ is called a \emph{boundary regular point} for $T$ if there is an $r > 0$ and a $k$-dimensional connected oriented embedded $C^{3, \af}$ submanifold $\Sm \subset B_r^n(p)$ without boundary in $B_r^n(p)$ such that $\spt (T \llcorner B_r^n(p)) \subset \Sm$. The set of boundary regular points is denoted by $\reg_b(T)$, while the set of \emph{boundary singular points} is defined by $\sing_b(T) := \Gm \setminus \reg_b(T)$.
\end{defn}

Fix $p \in \reg_b(T)$ (so $p \in \Gm)$ and let $\Sm$ be as in \Cref{regpt}. Choose $r >0$ so small that $B^n_r(p) \cap \Sm$ is diffeomorphic to a $k$-dimensional disk. Then the Constancy Theorem implies:
\begin{enumerate}
    \item $\Gm \cap B^n_r(p) \subset \Sm$ and divides $\Sm$ into two disjoint $k$-dimensional oriented connected embedded $C^{3, \af}$ submanifolds $\Sm^{\pm}$ with $\pl \Sm^\pm = \pm \Gm$.
    \item There is a natural number $Q \in \BN$ such that 
    \[T \llcorner B^n_r(p) = Q[[\Sm^+]] + (Q - 1)[[\Sm^-]].\]
\end{enumerate}
The number $Q$ is the \emph{multiplicity of $T$ at} $p \in \reg_b(T)$. The \emph{density of $T$ at $p \in \reg_b(T)$} is
\[\Theta_T(p) := Q - \frac{1}{2},\]
and coincides with the definition \eqref{dense} by \cite[Theorem 3.2]{CDHM}. The points $p$ at which $Q = 1$ are called \emph{density $\frac{1}{2}$ points}, or \emph{one-sided points}. The term ``one-sided" comes from the fact that $T \llcorner B^n_r(p) = [[\Sm^+]]$. In other words, $T$ can be locally identified with $\Sm^+$ near $p$. If $Q > 1$, we say that $p$ is a \emph{two-sided point}. See \cite[Example 1.3]{CDHM} for a helpful illustration of one-sided and two-sided boundary regular points. Since we intend to utilize PDE techniques, we need conditions under which a relatively open subset of one-sided regular points exists along $\Gm$. 

The question of the existence of one-sided boundary regular points for suitably regular $\Gm$ dates back to Almgren in his Big Regularity Paper \cite[Section 5.23, p. 835]{Alm1}. A classic result of Hardt and Simon \cite[Theorem 11.1]{HS1} says that if $ k = n - 1$, then $\reg_b(T) = \Gm$. However, without additional hypotheses we cannot be sure that a given point $p \in \reg_b(T)$ is one-sided for general $ 1 \leq k \leq n-2$. The first positive results on the existence of one-sided regular points in the higher codimension case were proved recently by De Lellis, De Philippis, Hirsch, and Massaccesi \cite[Theorem 1.6 \& Corollary 1.10]{CDHM}. We record the most relevant. 
\begin{thm}[{cf. \cite[Theorem 1.6 \& Corollary 1.10]{CDHM}}]\label{thm1}
    Assume $U\subset \BR^n$ and $W \subset \subset U$ are open connected sets. Suppose that $T \in \CI_{k,c}(\BR^n)$ and $\Gm \subset W$ satisfy \Cref{A1}. Then:
    \begin{enumerate}
    \item $\reg_b(T)$ is open and dense in $\Gm$.
    \item Every point in $\reg_b(T)$ is one-sided.
    \item $\reg_i(T)$ is connected.
    \item If $p \in \reg_b(T)$, then there is an $r > 0$ such that $T \llcorner B_r^n(p) = [[\Sm_u]]$ for some stationary solution $u \in C^{3, \af}(\overline{\Om}; \BR^{n-k}) \cap C^\om(\Om; \BR^{n-k})$ to \eqref{MSS}, where $\Om \subset \BR^k$ is a bounded, simply connected, $C^{3, \af}$ domain. 
    \item $\theta_T(p) = \Theta_T(p) = 1$ for all $p \in \reg_i(T)$ and $\BM(T) = \CH^k(\reg_i(T))$.
\end{enumerate}
\end{thm}

Unique continuation from Cauchy data for area-minimizing currents now follows by combining \Cref{thm1}, \Cref{analyticconti}, \Cref{obs1}, and Almgren's interior regularity theorem.
\begin{prop}[Unique Continuation from Cauchy Data]\label{UCP1}
    Assume $U\subset \BR^n$ and $W \subset \subset U$ are open connected sets. Let $T \in \CI_{k,c}(\BR^n)$ and $\Gm \subset W$ be as in \Cref{A1}. Suppose that $T^\prime \in \CI_{k,c}(\BR^n)$ is also area-minimizing in $U$ with $\spt T^\prime \subset W$, and $\pl T^\prime = [[\Gm]]$. If the (oriented) approximate tangent spaces for $\spt T$ and $\spt T^\prime$ agree along a relatively open patch $\Gm^\prime \subset \Gm$, then $T = T^\prime$.
\end{prop}

We close Section 3 with some remarks on the unbounded and higher multiplicity cases.
\begin{rmk}\label{unbdd}
    Let $U \subset \BR^n$ be an unbounded open connected set. As \Cref{thm1}(1) (see also \cite[Theorem 1.6]{CDHM}) still holds, \Cref{UCP1} can of course be extended to integral currents $T \in \CI_k(\BR^n)$ with possibly unbounded supports and unbounded boundaries $\Gm$ which are area-minimizing in $U$ if we know in advance that \begin{enumerate}
        \item There is an open set of one-sided regular points in $\reg_b(T)$.
        \item $\reg_i(T)$ is connected.
        \item $\Theta_{T} \equiv 1$ on $\reg_i(T)$.
    \end{enumerate} By \Cref{integerdensity}, the third condition follows from the first two.
\end{rmk}

\begin{rmk}
    When $T$ has higher multiplicity on the boundary, i.e. $\pl T = Q [[\Gm]]$ for some $Q \in \BN \setminus \{1\}$, Fleschler and Resende \cite[Theorem C]{FR25} shows that $\reg_b(T)$ is open and dense in $\Gm$. In particular, when $\Gm$ is connected and both $\Gm$ and $T$ are compactly supported, $\reg_b(T)$ contains only one-sided points \cite[Proposition 2.12]{FR25}. Therefore, it seems possible to obtain a similar result as \Cref{UCP1} if we know in advance that \begin{enumerate}
        \item $T$ and $T^\prime$ have the same sheeted structure (see \cite[Definition 1.3]{FR25}) near one-sided points.
        \item $\reg_i(T)$ is connected.
        \item $\Theta_T \equiv Q$ on $\reg_i(T)$.
    \end{enumerate} As in \Cref{unbdd}, the last condition should hold automatically if the second one holds.
\end{rmk}

We want to stress that, as pointed out by Morgan, \Cref{UCP1} depends heavily on the boundary regularity theory and \Cref{obs1}. Using \Cref{mark}, we see that unique continuation from Cauchy data for area-minimizing currents holds whenever we know (1)--(3) in \Cref{unbdd} are satisfied and $\reg_b(T)$ is interpreted in the $C^{1,1}$ sense (i.e. $\Gm^\prime \in C^{1,1}$).

\section{\textbf{Uniqueness of compactly supported smoothly calibrated currents}}

In this section, \Cref{UOPP} is proved as a consequence of \Cref{UCP1}. 

\subsection{Proof of \texorpdfstring{\Cref{UOPP}}{TEXT}}

We start with two general lemmata that are only related to comass one differential forms. Suppose $e_1,\ldots, e_n$ is an orthonormal basis for $\BR^n$. Let $1 \leq k \leq n$. For each $i = 1,\ldots, k$, we define \[e_i \, \lrcorner \, (e_1 \wedge \cdots \wedge e_k) \coloneqq (-1)^{i-1} e_1 \wedge \cdots \w \widehat{e_i} \wedge \cdots \wedge e_k.\] 
Hence, for each $i$ we have
\[
    e_i \w (e_i \, \lrcorner \, (e_1 \wedge \cdots \wedge e_k)) = e_1 \w \cdots \wedge e_k. 
\]
The first important observation is the \emph{first cousin principle} (cf. \cite[p. 161: Exercise 8(a)]{Har1}).

\begin{lem}[First Cousin Principle] \label{FirstCousin}
    Suppose $\varphi \in \CE^k(\BR^n)$ has comass one and let $e_1,\ldots,e_n$ be an orthonormal basis for $\BR^n$. For each $i = 1,\ldots, n$ and $j = 1,\ldots, n-k$, set 
    \begin{align*}
        \xi &:= e_1 \wedge \cdots \wedge e_k \in G(k,\BR^n)\\
        \xi_{ij} &:= e_{k+j} \wedge (e_i \, \lrcorner \, (e_1 \wedge \cdots \wedge e_k)) \in G(k,\BR^n).
    \end{align*}
    If $\iip{\varphi_p}{\xi} = 1$ for some $p \in \BR^n$, then $\iip{\varphi_p}{\xi_{ij}} = 0$ for all $i$ and $j$. 
\end{lem}
\begin{proof}
 For any fixed $i, j$, define the vector field $v(t) \coloneqq \cos t \, e_i + \sin t \, e_{k+j}$ and consider
\begin{equation*}
f(t) \coloneqq |\vph_p(v(t) \w (e_{i} \, \lrcorner \, (e_1 \wedge \cdots \wedge e_k)))|^2 = \cos^2 t + (\iip{\vph_p}{\xi_{ij}})^2 \sin^2 t + 2 \iip{\vph_p}{\xi_{ij}} \cos t \, \sin t,
\end{equation*}
which is smooth in $t$ as $p$ is fixed. Since $\vph$ has comass one, $f$ has a local maximum at $t = 0$. It follows that $0 = f'(0) = 2 \iip{\vph_p}{\xi_{ij}}$. As $i, j$ are arbitrary, the proof is complete.
\end{proof}

We call the $\xi_{ij}$ the \emph{first cousins of} $\xi$. From \Cref{FirstCousin}, we obtain our key lemma.

\begin{lem}[Intersecting Planes Lemma]\label{IPL}
 Let $\varphi \in \CE^k(\BR^n)$ have comass one, fix $p \in \BR^n$, and suppose $\{e_1, \ldots, e_{k-1}\} \subset \BR^n$ is an orthonormal set. If 
 \[\eta := e_1 \wedge \cdots \wedge e_{k-1} \in G(k-1,\BR^n)\] 
 and $v \in \mathbb{S}^{n-1} \cap \eta^\perp$ makes 
 \begin{equation*}
     \eta_v := e_1 \wedge \cdots \wedge e_{k-1} \wedge v \in \mathcal{G}_p(\varphi), 
 \end{equation*}
 then $v$ is unique. In particular, if for some $p \in \BR^n$ the oriented $k$-planes $\xi_1,\xi_2 \in \mathcal{G}_p(\varphi)$ satisfy $\xi_1 \cap \xi_2 \in G(k-1,\BR^n)$, then $\xi_1 = \xi_2$.  
\end{lem}
\begin{proof}
    Let $\eta$ be as in the statement of the lemma. Suppose that $v_1, v_2 \in \mathbb{S}^{n-1} \cap \eta^\perp$ have been chosen so that $\eta_{v_1}, \eta_{v_2} \in \mathcal{G}_p(\varphi)$. Write $v_2 := c_1 v_1 + c_2 v_1^\perp$, where $v_1^\perp \in \eta_{v_1}^\perp$ and $c_i \in \mathbb{R}$ for each $i = 1, 2$. Set $e_k := v_1$ and choose $e_{k+1},\ldots, e_{n}$ so that $\{e_1,\ldots, e_n\}$ is an orthonormal basis for $\BR^n$. Then 
    \[v_1^\perp = \sum_{j = 1}^{n-k} a_{j} e_{k+j} \] 
    for some $\{a_j \in \BR\}_{j=1}^{n-k}$.
    Since $\eta_{v_1}, \eta_{v_2} \in \mathcal{G}_p(\varphi)$, \Cref{FirstCousin} shows
    \begin{equation*}
    1 = \iip{\vph_p}{\eta_{v_2}} = c_1 + c_2 \sum_{j = 1}^{n-k} a_{j} \iip{\vph_p}{\xi_{kj}} = c_1,
\end{equation*} where $\xi_{kj}$ are first cousins of $e_1 \w \cdots \w e_k = \eta_{v_1}$. Since $v_2 \in \mathbb{S}^{n-1}$, we see that $v_2 = v_1$. This proves the first statement, and the second is immediate from the first.
\end{proof}

Given a smoothly calibrated integral current with compact support, all the compactly supported area-minimizing integral currents which have the same boundary must also be calibrated by the same form. This is proved using \Cref{closed=exact} and \Cref{calibrated}.

\begin{lem} \label{areaminimpliescal}
   Assume $U\subset \BR^n$ and $W \subset \subset U$ are open sets (not necessarily connected). Let $\vph \in \CE^k(\BR^n)$ be a calibration. Suppose that $T \in \CI_{k, c}(\BR^n)$ is calibrated by $\vph$ in $U$ with $\spt T \subset W$. If $T^\prime \in \CI_{k, c}(\BR^n)$ is area-minimizing in $U$ with $\pl T^\prime = \pl T$ and $\spt T^\prime \subset W$, then $T^\prime$ is also calibrated by $\vph$ in $U$.
\end{lem}

\begin{proof}
    Note that $T - T^\prime$ has compact support in $W \subset \subset U$ and $\pl (T - T^\prime) = 0$. Since $T^\prime$ is area-minimizing in $U$ and $T$ is calibrated by $\vph$, \Cref{calibrated} leads to \begin{equation} \label{eq_1}
        (T^\prime \llcorner W)(\vph) \leq \BM_W(T^\prime) \leq \BM_W(T^\prime+ (T-T^\prime) ) = \BM_W(T) = (T \llcorner W)(\vph) = T(\vph).
    \end{equation} Next, by \Cref{closed=exact}, there exists $S \in \CI_{k+1,c}(\BR^n)$ such that $\pl S = T-T^\prime$. Then \begin{equation} \label{eq_2}
        T(\vph) = (T^\prime + \pl S)(\vph) = T^\prime(\vph) + S(d \vph) = (T^\prime \llcorner W)(\vph)
    \end{equation} since $d \vph = 0$. Combining \eqref{eq_1} and \eqref{eq_2}, we see that $\BM_W(T^\prime) = (T^\prime \llcorner W)(\vph)$, which implies that $T^\prime$ is calibrated by $\vph$ in $W$ by \Cref{calibrated}. Finally, since $\|T^\prime\|(U \setminus W) = 0$, we conclude that $T^\prime$ is calibrated by $\vph$ in $U$.
\end{proof}

The main theorem now follows easily.

\begin{prop} \label{main}
    Assume $U\subset \BR^n$ and $W \subset \subset U$ are open connected sets. Suppose that $T \in \CI_{k,c}(\BR^n)$ and $\Gm \subset W$ are as in \Cref{A1}. Let $\vph \in \CE^k(\BR^n)$ be a calibration. Assume that $T$ is calibrated by $\vph$ in $U$. If $T^\prime \in \CI_{k, c}(\BR^n)$ is area-minimizing in $U$ with $\pl T^\prime = [[\Gm]]$ and $\spt T^\prime \subset W$, then $T^\prime = T$.
\end{prop}

\begin{proof}
    According to \Cref{areaminimpliescal}, $T^\prime$ is also calibrated by $\vph$ in $U$. By \Cref{thm1}(1) and (2), there exist $p \in \Gm$ and $r > 0$ such that $\Gm_r := B^n_r(p) \cap \Gm \subset \reg_b(T) \cap \reg_b(T^\prime)$ consists of one-sided regular boundary points. Let $\Sm$ and $\Xi$ be $k$-dimensional smooth submanifolds of $\BR^n$ satisfying $\spt T \cap B^n_r(p) = \Sm^+$ and $\spt T^\prime \cap B^n_r(p) = \Xi^+$, which can be done due to the discussion following \Cref{regpt}. Then for all $q \in \Gm_r$, the approximate tangent spaces for $\spt T$ and $\spt T^\prime$ at $q$ match $T_q \Sm^+$ and $T_q\Xi^+$, respectively. Moreover, since $\vph \in \CE^k(\BR^n)$ and the orientation $\xi_{\Sm^+}$ is $C^2$ in $\Sm^+$, the function $m \mapsto \iip{\vph_m}{\xi_T(m)}$ is continuous in $\Sm^+$. Therefore, $T_q \Sm^+ \in \CG(\vph)$ for all $q \in \Gm_r$. The same reasoning applied to $\Xi^+$ shows that $T_q \Xi^+ \in \CG(\vph)$ for all $q \in \Gm_r$ as well. Since $\pl \Sm^+ = \pl \Xi^+ = \Gm_r$, the intersection $T_q \Sm^+ \cap T_q \Xi^+$ must contain the $(k-1)$-dimensional subspace $T_q\Gm_r$ for each $q \in \Gm_r$. Then \Cref{IPL} shows that $T_q \Sm^+ = T_q \Xi^+$ for all $q \in \Gm_r$. Applying \Cref{UCP1}, we deduce that $T \llcorner U = T^\prime \llcorner U$. Since $\spt T \subset U$ and $\spt T^\prime \subset U$, we conclude that $T = T^\prime$.
\end{proof}

\begin{rmk}\label{final}
    In the special case $U = \BR^n$, we simply choose $W \subset \subset \BR^n$ satisfying $\spt T \cup \spt T^\prime \subset W$ to obtain \Cref{UOPP}.
\end{rmk}

As a consequence of \Cref{main} and \Cref{final}, we have proved uniqueness in the oriented Plateau problem for many important classes of area-minimizing submanifolds (\Cref{graph_calibrated}--\Cref{LOC}). We conclude with the sharpness of our hypotheses and extensions.

\subsection{Sharpness of hypotheses} We discuss extensions of our main theorems, as well as counterexamples to uniqueness when the hypotheses in \Cref{A1} are relaxed. 

\subsubsection{Connectedness of the boundary}
 Suppose $\Gm$ satisfies the hypotheses in \Cref{A1}, except that it is disconnected. Let $T \in \CI_{k,c}(\BR^n)$ be area-minimizing with $\pl T = [[\Gm]]$. The only potential issue is that of regularity. Indeed, for our argument to work, we must be able to guarantee that every connected component of $\reg_i(T)$ meets $\reg_b(T)$ in a relatively open set of one-sided boundary regular points for $T$. Then, since \Cref{areaminimpliescal} holds for disconnected boundaries, the first cousin principle (\Cref{FirstCousin}) and the intersecting planes lemma (\Cref{IPL}) force agreement of the approximate tangent planes along relatively open portions of $\Gm$ consisting of one-sided regular points that meet each connected component of $\reg_i(T)$. Hence, \Cref{obs1} applies and unique continuation from Cauchy data is propagated throughout the entirety of $\reg_i(T)$ by \Cref{analyticconti}.

\subsubsection{Compact support}
The ``compact support assumption" is crucial. The simplest counterexample is half-planes which have boundary equal to the axis but are at a non-zero angle. For example, consider $\Sm \coloneqq \{(x, 0, z) : x \geq 0, z \in \BR\} \subset \BR^3$ and $\Xi \coloneqq \{(0, y, z) : y \geq 0, z \in \BR\} \subset \BR^3$. Note that $[[\Sm]]$ is calibrated by $dx \w dz$ in $\BR^3$ and $[[\Xi]]$ is calibrated by $dy \w dz$ in $\BR^3$, so they are both globally area-minimizing and share the same boundary $[[z\text{-axis}]]$. Though $\Sm$ and $\Xi$ are diffeomorphic as smooth manifolds, $[[\Sm]]$ and $[[\Xi]]$ are certainly not the same as integral currents. 
    
In $\BR^3$, the half-plane, half of the helicoid, and half of the Enneper surface, each positioned so that they share the same boundary line, serve as another counterexample since they are all globally area-minimizing (see \cite{Whi96}, \cite{Per07}, or \cite[Example 1.1]{EW22}). However, they are neither equal as integral currents nor diffeomorphic to one another as smooth manifolds with boundary. Notice that, in each case, it is \Cref{areaminimpliescal} that fails. The key point is that the compact support assumption restricts the angle that the tangent planes at the boundary of a calibrated integral current can form with the support of the boundary current, thereby ensuring uniqueness if we have sufficient boundary regularity.

\subsubsection{Integral flat chains}

As \Cref{FirstCousin} and \Cref{IPL} are pointwise statements, they hold for any comass one $k$-form defined at $p \in\BR^n$, without requiring any regularity assumptions. On the other hand, the proof of \Cref{main} only uses the fact that the calibration is a $k$-form that is continuous near the boundary. Therefore, it is natural to expect that \Cref{main} may still hold for non-smooth calibrations.

Using the language of \emph{flat $k$-chains}, one can weakly define the notion of calibrations. Loosely, flat $k$-chains can be viewed as a generalization of compactly supported normal $k$-currents (cf. \cite[Section 4.1.12]{Fed69}). \emph{Flat $k$-cochains} are bounded real-valued linear functionals on flat $k$-chains (see \cite[Section 4.1.19]{Fed69}). They can be regarded as $\CH^n$-measurable $k$-forms. A \emph{calibrating flat $k$-cochain} is a (weakly) closed flat $k$-cochain that has comass one. 

The definitions above, together with the fact that the homology of flat chains is isomorphic to singular homology (see \cite[Section 4.4.5--4.4.6]{Fed69}), shows that \Cref{calibratedimpliesmin} and \Cref{areaminimpliescal} generalize naturally to the setting of flat chains. More precisely, calibrated flat chains are area-minimizing among all flat chains with the same boundary. Furthermore, if $T$ is a flat $k$-chain calibrated by a flat $k$-cochain $\af$ and $T^\prime$ is an area-minimizing flat $k$-chain such that $\pl T^\prime = \pl T$, then $T^\prime$ must also be calibrated by $\af$. On the other hand, an integral flat $k$-chain whose mass and boundary mass are finite is, in fact, an integral $k$-current with compact support (see \cite[Section 4.2.16]{Fed69}). Hence, the boundary regularity theory developed by De Lellis, De Philippis, Hirsch, and Massaccesi in \cite{CDHM} applies to such integral flat $k$-chains. We can therefore generalize \Cref{UOPP} to integral flat chains. 

\begin{thm}[Uniqueness for Calibrated Integral Flat Chains] \label{flatchain}
    Suppose that $\Gm \subset \BR^n$ is as in \Cref{A1} and let $\af$ be a calibrating flat $k$-cochain of $\BR^n$. Further suppose that there is an open set $\mathscr{O} \subset \BR^n$ with $\mathscr{O} \cap \Gm \neq \emptyset$ such that $\alpha \lvert_{\mathscr{O}}$ can be identified with a $k$-form $\vph_\alpha \in C^0(\mathscr{O}; \Lambda^k(\BR^n))$. Assume that $T$ is an integral flat $k$-chain in $\BR^n$ of finite mass such that $\pl T = [[\Gm]]$, and that $T$ is calibrated by $\af$ in $\BR^n$. If $T^\prime$ is another globally area-minimizing integral flat $k$-chain in $\BR^n$ with $\pl T^\prime = [[\Gm]]$, then $T^\prime = T$.
\end{thm}

The continuity assumption on the calibrating flat cochain in \Cref{flatchain} is essential. To illustrate this, we consider two examples. The first is the ``four corners example" described in the second paragraph of the introduction. 

\begin{ex} \label{corner}
    Let $A = (1, 1)$, $B = (1, -1)$, $C = (-1, -1)$, and $D = (-1, 1)$ be points in $\BR^2$. Assign the negative orientation to $A$ and $C$, and the positive orientation to $B$ and $D$. Then 
    \begin{equation}
    [[(A, B)]] + [[(C, D)]] \quad  \text{and} \quad [[(A, D)]] + [[(C, B)]] \label{abcd}
    \end{equation}
    are two area-minimizing integral flat $1$-chains with boundary $- [[A]] + [[B]] - [[C]] + [[D]]$. Define \[\af := \left\{ \begin{aligned}
    - dy& \quad \text{in } \Int(\triangle ABO), \\
    dx& \quad \text{in } \Int(\triangle BCO),  \\
    dy& \quad \text{in } \Int(\triangle CDO), \\
    - dx& \quad \text{in } \Int(\triangle DAO), \end{aligned}\right.\] where $O = (0, 0)$ is the origin. That $\af$ has comass one can be readily verified. Although $\alpha$ is discontinuous along the diagonals of the square determined by $A,B,C,D$, it is closed in the sense of flat cochains due to its symmetry---for instance, in the first quadrant the diagonal bisects the angle between $-dx$ and $-dy$. Hence, $\af$ is a calibrating flat $1$-cochain. Moreover, both flat $1$-chains in \eqref{abcd} are calibrated by $\af$, demonstrating non-uniqueness when $\alpha$ is not continuous at any boundary point.
\end{ex}

\Cref{baseball} below further demonstrates how symmetric boundaries can lead to non-uniqueness, and can be verified experimentally using soap film wire frame experiments. 
\begin{ex}[A Baseball Seam]\label{baseball}
    Let $\Gm$ be the curve in $\BS^2$ smoothly tracing the seam of a baseball oriented positively. Then there are two oriented area-minimizing surfaces $\Sm$ and $\Xi$ in $\BR^3$ with boundary $\Gm$ which are diffeomorphic by rigid motions. The surfaces $\Sm$ and $\Xi$ can be represented as inward deformations to $\BS^2$ of the two connected faces of $\BS^2$ bounded by $\Gm$. By \Cref{flatchain}, $[[\Gm]]$ is not the boundary of a compactly supported ``continuously calibrated" integral flat chain $T$. This can also be proved directly. Indeed, if $[[\Sm]]$ is calibrated by a continuous $\alpha$ (as a $2$-form), then so is $[[\Xi]]$. However, the approximate tangent planes to $[[\Sm]]$ and $[[\Xi]]$ along $\Gm$ meet at a non-zero angle along the entirety of $\Gm$, violating \Cref{IPL}.
\end{ex}

When $k = n-1$ (i.e. in codimension one), Federer \cite[Section 4.12--4.13]{Fed74} showed that every area-minimizing real flat chain is calibrated by some flat cochain. This result is quite strong; however, the construction is non-constructive as it relies on functional analytic tools like the Hahn--Banach theorem. In the same paper \cite[Section 5.10]{Fed74}, he also proved that every area-minimizing integral flat chain remains area-minimizing, even when compared against real flat chains. Combining these two results, it follows that every codimension one area-minimizing integral flat chain is calibrated by some flat cochain. Applying this and \Cref{flatchain} in tandem, we obtain uniqueness in the Plateau problem for hypersurfaces.

\begin{cor}[Uniqueness in Codimension One] \label{plateau}
     Suppose $\Gm \subset \BR^n$ satisfies the hypotheses of \Cref{A1}. Let $T$ be a globally area-minimizing integral flat $(n-1)$-chain of finite mass with $\pl T = [[\Gm]]$, and let $\alpha$ be a calibrating flat $(n-1)$-cochain for $T$. If there is an open set $\mathscr{O} \subset \BR^n$ with $\mathscr{O} \cap \Gm \neq \emptyset$ such that $\alpha \lvert_{\mathscr{O}}$ can be identified with an $(n-1)$-form $\vph_\alpha \in C^0(\mathscr{O}; \Lambda^{n-1}(\BR^n))$, then $T$ is the unique globally area-minimizing integral flat $(n-1)$-chain with boundary $[[\Gm]]$.
\end{cor}
\begin{ex}[Simons Cone] \label{SC}
    The Simons cone $\mathcal{S} \subset \BR^8$ is the hypercone with \emph{link}\footnote{See the paragraph preceding \Cref{cone}.} \[\BS^3\Big(\frac{\sqrt{2}}{2}\Big) \times \BS^3\Big(\frac{\sqrt{2}}{2}\Big) \Leftrightarrow \mathcal{S} = \{(x, y) \in \BR^4 \times \BR^4 : |x| = |y|\}.\] Its local stability was proved by Simons in \cite[Section 6]{Sim68}. Later, Bombieri, De Giorgi, and Giusti proved that it is globally area-minimizing by constructing a calibration that is singular only at the origin (see \cite[Lemma 1 \& Section 3]{BDG69}), so \Cref{plateau} applies. In fact, all homogeneous area-minimizing hypercones admit such calibrations by \cite[Theorem 1.9]{Zh16}.
\end{ex}

Observe that the $\Gm$ in \Cref{baseball} is a smooth simple closed curve so that $[[\Gm]]$, $[[\Sm]]$, and $[[\Xi]]$ satisfy the assumptions in \Cref{plateau}. The only issue is that, in every neighborhood of $\Gm$ in $\BR^n$, $[[\Sm]]$ and $[[\Xi]]$ cannot be calibrated by a continuous differential form in the sense of flat chains. In both \Cref{corner} and \Cref{baseball}, this is due to the symmetry of the boundary, suggesting that the regularity of the calibrating form along the boundary is related to the boundary geometry. It is therefore of interest to determine what boundaries can be spanned by continuously calibrated area-minimizing integral flat chains. This will likely require methods for constructing singular calibrations. See \cite{BDG69, Zh16, Zh17} for examples.

\section{\textbf{General ambient manifolds}}
One can define calibrations on a general Riemannian manifold. A Riemannian manifold together with a smooth calibration is called a \emph{calibrated manifold}. There are many examples of calibrated manifolds (see e.g. \cite{HL82}). However, in contrast to \Cref{calibratedimpliesmin}, a calibrated current in a calibrated manifold can only be assumed to be \emph{homologically} area-minimizing (see \cite[Chapter II, Theorem 4.2]{HL82}). Nonetheless, all the results in Section 3.2--4 and Appendix A carry over to a general ambient manifold $M$ that is complete without boundary, analytic, and satisfies $H_k(M; \BR) = 0$. 

The last two conditions are essential. For analyticity, our arguments heavily rely on the fact that $C^1$ minimal submanifolds in an analytic ambient manifold are actually analytic by the bootstrapping argument (see also \Cref{analytic}), so that \Cref{analyticconti} holds. For the homological condition, a general ambient manifold $M$ may contain a closed $k$-dimensional submanifold $\Sm$ calibrated by a $k$-form $\vph$. For example, $\Sm = \BS^n$ is a closed special Lagrangian submanifold in $T^*\BS^n$ with the Stenzel metric (see e.g. \cite{Ste93} or \cite[Section 5.1]{TW18}).

By chopping $\Sm$ into two submanifolds $\Sm^+$ and $\Sm^-$, we note that both are calibrated by $\vph$ and share a common boundary having opposite orientations. Therefore, $[[\Sm^+]]$ and $-[[\Sm^-]]$ share the same oriented boundary and are calibrated by $\vph$ and $- \vph$, respectively. If the volume of $\Sm$ is divided unequally, we get a counterexample to \Cref{calibratedimpliesmin}. On the other hand, if the volume is divided equally, we obtain a counterexample to \Cref{areaminimpliescal} and to uniqueness in the oriented Plateau problem. To exclude these examples, we must assume $H_k(M; \BR) = 0$. With this assumption, statements similar to \Cref{calibratedimpliesmin} and \Cref{areaminimpliescal} hold. The reason is that we rule out the existence of closed calibrated $k$-dimensional submanifolds in $M$ by the universal coefficient theorem and de Rham theorem.

\appendix
\section{\textbf{Uniqueness of restrictions of global area-minimizers}}

We show that if $T \in \CI_{k, c}(\BR^n)$ is a global area-minimizer arising as the restriction of another global area-minimizer $T^\prime \in \CI_{k, c}(\BR^n)$, then $T$ is the unique global area-minimizer among all compactly supported integral currents with the same boundary as $T$. The authors believe that it must be known to experts in the field, particularly when stated as \Cref{cone}. 

\begin{thm} \label{main_2}
    Let $T \in \CI_{k, c}(\BR^n)$ be globally area-minimizing. Suppose that there exists a globally area-minimizing integral $k$-current $T^\prime \in \CI_{k, c}(\BR^n)$ with compact support such that $T^\prime \llcorner \spt T  = T$ and $\spt (\pl T) \subset \spt T^\prime \setminus \spt (\pl T^\prime)$. Then $T$ is the unique global area-minimizer among all compactly supported integral $k$-currents with the same boundary as $T$.
\end{thm}

\begin{proof}
    Suppose that $S \in \CI_{k,c}(\BR^n)$ is another globally area-minimizing integral $k$-current with compact support that satisfies $\pl S = \pl T$. Then the assumptions on $T^\prime$ show that \begin{equation} \label{eq1}
        \BM(T^\prime) \leq \BM(T^\prime - T + S) \leq \BM(T^\prime - T) + \BM(S) \leq \BM(T^\prime - T) + \BM(T) = \BM(T^\prime).
    \end{equation} In other words, $T^\prime - T + S$ is a global area-minimizing integral $k$-current with compact support which has the same boundary as $T^\prime$.

    We next claim that $\spt T^\prime$ and $\spt (T^\prime - T + S)$ agree $\CH^k$-a.e. on $\spt T^\prime \setminus \spt T$. Suppose otherwise. Then there exists a relatively open, non-empty subset $\mathscr{O}$ of $\spt T^\prime \setminus \spt T$ such that $S$ cancels out $T^\prime - T$ on $\mathscr{O}$. That is, $S \llcorner \mathscr{O} = -(T^\prime - T)\llcorner \mathscr{O} = -T^\prime \llcorner \mathscr{O}.$ We then compute the mass \[\begin{aligned}
    \BM(T^\prime - T + S) &= \BM(T^\prime - T + S + S\llcorner\mathscr{O} - S\llcorner\mathscr{O}) \\
    &\leq \BM(T^\prime - T - T^\prime \llcorner \mathscr{O}) + \BM(S - S\llcorner\mathscr{O}) \\
    &< \BM(T' - T) + \BM(S) \\
    &= \BM(T^\prime).\end{aligned}\] This contradicts \eqref{eq1}, and hence the claim follows.

    For any $x \in \reg_i(T^\prime) \cap \spt (\pl T)$, let $M_x$ be the connected component of $\reg_i(T^\prime)$ containing $x$. Then the aforementioned claim together with \Cref{analyticconti} shows that \[M_x = \reg_i((T^\prime - T + S) \llcorner M_x).\] Moreover, by \Cref{integerdensity} and the assumption that $\pl T = \pl S$, $T^\prime$ and $T^\prime - T + S$ have the same densities on $M_x$. Taking the union over all such $M_x$, we conclude that $T^\prime = T^\prime- T + S$ on an open neighborhood of $\reg_i(T^\prime) \cap \spt (\pl T)$ in $\spt T^\prime$. In other words, there exists an open neighborhood $\mathscr{U} \subset \spt T$ of $\reg_i(T^\prime) \cap \spt (\pl T)$ such that \[T \llcorner \mathscr{U} = S \llcorner \mathscr{U}.\]

    Finally, we claim that every connected component of $\reg_i(T)$ and $\reg_i(S)$ must intersect $\mathscr{U}$ non-trivially. If not, we may assume that there is a connected component $N$ of $\reg_i(T)$ such that $N \cap \mathscr{U} = \emptyset$. Note that both $\sing_i(T)$ and $\sing_i(T^\prime) \cap \spt (\pl T)$ have Hausdorff dimension at most $k-2$. Then a cutoff argument---essentially the same as that used in the third paragraph of the proof of \cite[Lemma 4.4]{ELV24}, with $d-5$ and $p < 5$ replaced by $k-2$ and $p < 2$, respectively---shows that $\pl(T \llcorner N) = 0$. This violates the assumption that $T$ is area-minimizing. The same reasoning applies to $S$, and thus the claim follows. With this claim established, the unique continuation argument---by using \Cref{analyticconti} and \Cref{integerdensity}---in the previous paragraph can therefore be extended to the whole $\reg_i(T^\prime) \cap \spt T$, yielding \[T \llcorner (\reg_i(T^\prime) \cap \spt T) = S \llcorner (\reg_i(T^\prime) \cap \spt T).\] As $\sing_i(T^\prime)$ is $\CH^k$-null, the proof is complete.
\end{proof}

\begin{rmk}
    In contrast to \Cref{main}, \Cref{main_2} does not impose any assumptions on the regularity, connectivity, or multiplicity of $\pl T$.
\end{rmk}

Let $C \subset \BR^n$ be a $k$-dimensional cone. Then $C$ is called \emph{regular} if $C \setminus \{p\} = \reg_i(C)$ for some $p \in \BR^n$. For such a $p$, we call $C \cap \pl B^n_1(p)$ its \emph{link}. We have the following corollary for area-minimizing regular cones (e.g. \Cref{LOC} or \Cref{SC}):
\begin{cor}\label{cone}
   Let $C$ be a $k$-dimensional regular cone in $\BR^n$, and suppose that $[[C]]$ is globally area-minimizing in $\BR^n$. Then $[[C]] \llcorner B^n_1(p)$ is the unique global area-minimizer bounded by its link.
\end{cor}

Finally, we note that \Cref{graph_calibrated}, \Cref{UOPP}, and \Cref{main_2} together imply the uniqueness statement in \cite[Theorem 2]{Ha77} when the boundary of the graph is at least $C^{3,\alpha}$.


\bibliographystyle{abbrv}
\bibliography{references}

\end{sloppypar}
\end{document}